\documentclass[12pt, a4paper]{amsart}
\usepackage[english]{babel}
\usepackage{amsmath,enumerate, amsfonts, amssymb,amsthm, xypic}
\usepackage[dvips]{graphics} 
\usepackage{color}

\newtheorem{thm}{Theorem}[section]

\newtheorem{prop}[thm]{Proposition}
\newtheorem{cor}[thm]{Corollary}

\theoremstyle{definition}

\newtheorem{rmk}[thm]{Remark}

\DeclareMathOperator{\R}{\mathbf R}

\DeclareMathOperator{\PP}{\mathbf P}

\DeclareMathOperator{\Int}{Int}

\DeclareMathOperator{\Sing}{Sing}

\DeclareMathOperator{\eps}{\epsilon}

\title{Virtual Poincar\'e polynomial of the link of a real algebraic variety}
\author{Goulwen Fichou and Masahiro Shiota}
\thanks{The first author has been supported by the ANR project ANR-08-JCJC-0118-01.}
\address{IRMAR (UMR 6625), Universit\'e de Rennes 1, Campus de
  Beaulieu, 35042 Rennes Cedex, France}
\address{Graduate School of Mathematics, Nagoya University, Chikusa, Nagoya, 
464-8602, Japan}
\date\today

\begin{document}

\begin{abstract} The Euler characteristic of the link of a real algebraic variety is an interesting
  topological invariant in order to discuss local topological
  properties. We prove in the paper that an invariant stronger than
  the Euler Characteristic is well defined for the link of an
  algebraic variety: its virtual Poincar\'e polynomial.
\end{abstract}
\maketitle

\section*{Introduction}
Let $X\subset \R^n$ be a real algebraic variety and $a$ a point in
$X$. Let $S(a,\eps)$ denote the sphere with centre $a$ and radius
$\eps >0$ in $\R^n$. The local conic structure theorem implies that, for
$\eps$ small enough, the topological type of $X\cap S(a,\eps)$ is
independent of $\eps$; this intersection is called the link of $X$ at
$a$, denoted by $lk(X,a)$. The link is known to be a very interesting
topological invariant, as illustrated by Sullivan Theorem \cite{Su}
that asserts that the Euler characteristic of the link at any point of
a real algebraic subvariety of $\R^n$ is even. This evenness even
leads to a topological characterisation of affine real algebraic varieties in
dimension less than two. In higher dimension only necessary conditions
are known (see \cite{AK,MCPlink}).

Actually, the link is not only a topological invariant but also a
semialgebraic topological invariant: for $\eps,\eps'$ small enough one
may find a semialgebraic homeomorphism between $X\cap S(a,\eps)$ and
$X\cap S(a,\eps')$ (see \cite{CK} for a more precise statement). But can
we find a stronger invariance than the semialgebraic topological one for the link of a real algebraic
subvariety? The link is again a real algebraic variety, so one may
wonder whether the link is an algebraic or Nash invariant
(Nash means here analytic and semialgebraic).

In this paper, we focus on a finer invariant for real algebraic
varieties than the Euler characteristic: the virtual Poincar\'e
polynomial. The virtual Poincar\'e polynomial, denoted by $\beta$, has been introduced in
\cite{MCP} as an algebraic invariant for Zariski constructible real algebraic sets. Then the invariance of the
virtual Poincar\'e polynomial has been established under Nash
diffeomorphisms \cite{Fi}, and finally in \cite{MCP2} under regular
homeomorphisms. We prove in Corollary \ref{link} that the virtual Poincar\'e
polynomial of the link of an affine real algebraic variety is well
defined. More generally, we state in Theorem \ref{main}) that the
virtual Poincar\'e polynomial of the fibres of a regular map between
real algebraic varieties is
constant along the connected components of an algebraic stratification of
the target variety.

The proof of Theorem \ref{main} is different from the proof of the
fact that the Euler characteristic of the link is well defined. If one
knows that the link is invariant under semialgebraic homeomorphisms,
we are not able to prove that this is still the case under the stronger
class of regular homeomorphisms. To reach our aim, we pass through
resolution of singularities \cite{H} and a crucial property of the
virtual Poincar\'e polynomial. Similarly to the Euler characteristic (with compact
supports), the virtual Poincar\'e polynomial is an additive
invariant: $\beta(X)=\beta(Y)+\beta(X\setminus Y)$ for $Y\subset X$ a
closed subvariety of $X$. This enables to express the virtual
Poincar\'e polynomial of the link at a point of the algebraic variety
$X$ as a sum of terms coming from a
given resolution of the singularities of $X$. Then
using Nash triviality results, in the smooth case \cite{CS} together
with the normal crossing case \cite{FKS} and applying the invariance of
the virtual Poincar\'e polynomial under Nash diffeomorphisms to selected
varieties coming from the resolution process enable to reach our goal.

\vskip 5mm

The paper is organised as follows. In the first section we recall the
basic properties of the virtual Poincar\'e polynomial as we need
it. We adapt also the Nash triviality results from \cite{CS} and
\cite{FKS} to our particular setting. In the second section we state
the invariance of the virtual Poincar\'e polynomial of the fibres of a
regular map between real algebraic varieties, and deduce from it
Corollary \ref{link}.
 

\section{Preliminaries}\label{}
\subsection{Virtual Poincar\'e polynomial}
For real algebraic varieties, the best additive and multiplicative
invariant known is the virtual Poincar\'e polynomial
\cite{MCP}. It assigns to a Zariski constructible real algebraic set a polynomial
with integer coefficients in such a way that the coefficients coincide
with the Betti numbers with $\mathbb Z_2$-coefficients for proper non singular
real algebraic varieties.

\begin{prop}(\cite{MCP})\label{beta} Take $i\in \mathbb N$. The Betti
 number $\beta_i(\cdot)=\dim H_i(\cdot, \frac {\mathbb Z}{2 \mathbb
   Z})$, considered on compact non singular real algebraic varieties,
 admits an unique extension as an additive map $\beta_i$ to the
 category of Zariski constructible real algebraic sets, with values in $\mathbb Z$. Namely
$$\beta_i(X)=\beta_i(Y)+\beta_i(X \setminus Y)$$
for $Y \subset X$ a closed subvariety of $X$. 

Moreover the polynomial
$\beta(\cdot)=\sum_{i \geq 0} \beta_i(\cdot)u^i \in \mathbb Z [u]$ is
multiplicative: 
$$\beta(X\times Y)=\beta(X)\beta(Y)$$ 
for Zariski constructible real algebraic sets.
\end{prop}

The invariant $\beta_i$ is called the i-$th$ virtual Betti number, and the
polynomial $\beta$ the virtual Poincar\'e polynomial. By evaluation of the virtual Poincar\'e polynomial at $u=-1$ one recovers the Euler characteristic with compact supports \cite{MCP}.

\begin{prop}\label{beta-nash}(\cite{Fi,MCP2}) Let $X$ and $Y$ be Nash diffeomorphic compact real algebraic
 varieties. Then their
 virtual Poincar\'e polynomial coincide.
\end{prop}

\begin{rmk} In the non compact case the result is no longer true. For
 example an hyperbola in the plane in Nash diffeomorphic to the union
 of two lines, but the virtual Poincar\'e polynomial of the hyperbola
 is $u-1$ whereas the virtual Poincar\'e polynomial of the two lines is $2u$.
\end{rmk}

A crucial property of the virtual Poincar\'e polynomial is that the
degree of $\beta(X)$ is equal to the dimension of the Zariski constructible real algebraic set $X$. In particular, and
contrary to the Euler characteristics with compact supports, the
virtual Poincar\'e polynomial cannot be zero for a non-empty set.

\subsection{Nash triviality}
The Nash triviality Theorem of M. Coste and
M. Shiota \cite{CS} gives the local Nash triviality for proper Nash
maps with smooth fibres. We state it below in the form we need it later
 and show how to
derive it from Theorem A in \cite{CS}.

\begin{thm}\label{CS} Let $X$ and $Y$ be a real algebraic
 sets and $p: X \to Y$ be a regular map. Assume $X$ is
non singular. There exists a real algebraic
subset $Z\subset Y$ of strictly positive codimension in $Y$ such that
any two fibres of $p$ over a connected component of $Y\setminus Z$ are
Nash diffeomorphic.
\end{thm}

\begin{proof} Note that it suffices to work in the category of Nash
  manifolds and Nash mappings. Actually, if we remove from $Y$ its
  algebraic singular points $\Sing Y$, then $Y\setminus \Sing Y$
  becomes a Nash manifold. Moreover if we obtain a semialgebraic subset $Z'$ of $Y$ such that $p^{-1}(y_1)$ and $p^{-1}(y_2)$ are Nash diffeomorphic for any points $y_1$ and $y_2$ in one connected component of $Y-Z'$ then the Zariski closure $Z$ of $Z'$ 
fulfills 
the requirements.

Now we treat the case $\dim X\ge\dim Y$ and $\dim p(X)=\dim Y$ since
otherwise we can remove $p(X)$ from $Y$. We may assume that
$p(X)$ is equal to $Y$ by removing if necessary $\overline{p(X)}-\Int
p(X)$ from $Y$, where the closure and interior are those in $Y$. We
may assume moreover that $p$ is a submersion because the critical value set is semialgebraic and of smaller dimension by Sard's Theorem. 
Finally, we may assume that $X$ and $Y$ are included in
$\R^m\times\R^n$ and $\R^n$ respectively, and that $p$ is the restriction to $X$ of the projection $\R^m\times\R^n\to\R^n$ by replacing $X$ with the graph of $p$. 
Then $p:X\to Y$ satisfies the conditions of Theorem A in \cite{CS}. 
\end{proof}

Let $M$ be a Nash manifold and $X\subset M$ be a Nash subset of
$M$. Then $X$ has only normal crossing if for each $a\in M$ there
exists a local coordinate system $(x_1,\ldots,x_n)$ at $a\in M$ such
that $X$ is a union of some coordinate spaces in a neighbourhood of $a$
in $M$.

The Nash isotopy Lemma (Theorem I in \cite{FKS}) of T. Fukui, S. Koike and M. Shiota
extends the Nash triviality obtained in Theorem \ref{CS} to the normal
crossing situation. A subvariety $N$ of a non singular real algebraic variety $M$ has only normal crossings if at any point of $N$ there exists a local system of coordinates such that $N$ is a union of some coordinates spaces.

We state below the Nash isotopy Lemma as we need it 
later, 
and prove
how to derive it from Theorem I in \cite{FKS}.

\begin{thm}\label{FKS} Let $M$ be a non singular real algebraic variety and
 $N_1,\ldots,N_k$ be non singular real algebraic subvarieties of $M$ such that $\cup_{i=1}^kN_i$ has only
 normal crossing. Let $Y$ be a real algebraic variety and $p :M \to Y$ be
 a proper 
 regular 
 map. There exists a real algebraic
subset $Z\subset Y$ of strictly positive codimension in $Y$ such that
for any connected component $P$ of $Y\setminus Z$, there exist $y\in
P$ and a Nash diffeomorphism 
$$\phi:(M;N_1,\ldots,N_k) \to \big(M \cap
p^{-1}(y);N_1\cap p^{-1}(y),\ldots,N_k\cap
p^{-1}(y)\big)\times P$$
such that $p \circ \phi^{-1}:\big(M \cap p^{-1}(y)\big)
\times P \to P$ is the canonical projection.
\end{thm}

In particular, any two fibres of $p$ along $P$ are Nash diffeomorphic.

\begin{proof}[Proof of Theorem \ref{FKS}]
Similarly to the proof of Theorem \ref{CS}, there exists an algebraic
subvariety $Z \subset Y$ such that $Y\setminus Z$ is smooth and $p$
together with its restrictions to $N_{i_1}\cap \cdots \cap N_{i_s}$,
for $0\leq i_1<\cdots<i_s\leq k$, are submersions onto $Y\setminus
Z$. To apply Theorem I in \cite{FKS}, it remains to prove that the
connected components of $Y\setminus Z$ are Nash diffeomorphic to an open simplex.

Let $X$ be a bounded semialgebraic set included in $\R^n$. 
Then the triangulation theorem of semialgebraic sets states that there
exist a finite simplicial complex $K$ and a semialgebraic
homeomorphism from the underlying polyhedron $|K|$ to $\overline X$
such that $\pi^{-1}(X)$ is the union of some open simplexes in $K$
(Theorem 9.2.1 in \cite{BCR}). Moreover we can choose $\pi$ so that the restriction to each open simplex is a Nash embedding. 
Hence if we remove the union of $\pi(\sigma)$ for simplexes $\sigma$
in $K$ of dimension smaller than $\dim X$, then each connected
component of $X$ become Nash diffeomorphic to an open simplex.

In case $X$ is not bounded in $\R^n$, embed $\R^n$ in the sphere $S^n$
and then embed
$S^n$ in $\R^{n+1}$ by Nash embeddings, so that $X$ becomes bounded in
$R^{n+1}$ and we may apply the preceding argument.
\end{proof}

\subsection{Resolution of singularities}
The desingularisation Theorem of H. Hironaka \cite{H} (or
\cite{BM} for the form used in this paper) transforms a
singular variety into a non singular one together with normal crossings
with the exceptional divisors of the resolution. In view of the proof
of Theorem \ref{main} in section $2$, it will enable us
to use both Nash Triviality Theorem and Nash Isotopy Lemma after a
resolution of the singularities of the source space in the singular
case.

We recall that a smooth submanifold $N\subset M$ of a manifold $M$ and
a normal crossing divisor $D\subset M$ have
simultaneously only normal crossings if, locally, there exist local
coordinates such that $N$ is the intersection of some coordinate hyperplanes and $D$ is the union
of some coordinate hyperplanes.

\begin{thm}\label{hiro}(\cite{BM}) Let $M$ be a non singular real algebraic variety and
 $X\subset M$ be a real algebraic subvariety of $M$. There exists a finite
 sequence of blowings-up
$$\xymatrix{  M=M^s   \ar[r]^{\pi^{s-1}} &  M^{s-1} \ar[r]
 & \cdots \ar[r]^{\pi^{1}} & M^1=M}$$
with smooth centres $C^i\subset M^i$ such that, if $X^1=X$ and
$E^1=\emptyset$, and if we denote by $X^{j+1}$ the strict transform of
$X^j$ by $\pi^j$ and by $E^{j+1}$ the exceptional divisor
$E^{j+1}=(\pi^j)^{-1}(C^j\cup E^j)$, then 
\begin{enumerate}
\item each $C^j$ is included in $X^j$, and is of dimension smaller than $\dim X^j$,
\item each $C^j$ and $E^j$ have simultaneously only normal crossings,
\item $X^s$ is non singular,
\item $X^s$ and $E^s$ have simultaneously only normal crossings.
\end{enumerate}
\end{thm}

\begin{rmk} The fact that the centre $C^j$ together with the divisor
  $E^j$ have simultaneously only normal crossings at any stage of the
  resolution process will enable to use Theorem \ref{FKS} in the proof of Theorem \ref{main}.
\end{rmk}




\section{Virtual Poincar\'e polynomial of the link}\label{}


Next result states the constancy of the virtual Poincar\'e polynomial of the fibres
of a regular map along the connected components of an algebraic
stratification of the target space.

\begin{thm}\label{main} Let $X\subset \R^n$ and $Y\subset \R^m$ be a real algebraic
 sets and $f: X \to Y$ be a regular map. There exist a
 real algebraic subset $Z\subset Y$ of strictly positive codimension
 in $Y$ such that for any $a,b$ in the same connected component of
 $Y\setminus Z$, the virtual Poincar\'e of $f^{-1}(a)$ is equal to the virtual Poincar\'e of $f^{-1}(b)$.
\end{thm}

\begin{rmk} In the statement of Theorem \ref{main} we cannot avoid
  taking the connected components of $Z\setminus Y$, as illustrated by
  the square map from the line to itself. Indeed, the virtual
  Poincar\'e polynomial of the fibres is equal to $2$ along strictly
  positive real numbers and $0$ along strictly negatives ones.
\end{rmk}

Before giving the proof of Theorem \ref{main}, we apply it to define
the virtual Poincar\'e polynomial of the link at a point of a real algebraic variety.
Let $X\subset \R^n$ be a real algebraic variety and $a\in X$ be a point in
$X$. Let $S(a,\eps)$ denote the sphere with centre $a$ and radius
$\eps >0$ in $\R^n$. The set $lk(X,a)=X\cap S(a,\eps)$ is the link of $X$ at $a$.

\begin{cor}\label{link} Let $X\subset \R^n$ be a real algebraic variety and $a\in X$ a
 point in $X$. There exists $\delta >0$ such that for $\eps,
 \eps'>0$, the virtual Poincar\'e polynomial of $X\cap S(a,\eps)$ is equal to
 the virtual Poincar\'e polynomial of $X\cap S(a,\eps')$, provided that
 $\eps<\delta$ and $\eps'<\delta$.
\end{cor}

\begin{proof} We apply Theorem \ref{main} to the map $d_a:X \to \R$
  defined by taking the square of the
 Euclidean distance to $a$ in $\R^n$, namely $d_a(x)= ||x-a||^2$. Then
 there exist a finite number of points $-\infty=b_0<b_1<\cdots < b_r<b_{r+1}=+\infty$ in $\R$ such
 that the virtual Poincar\'e polynomial of the fibres of $d_a$ are
 constant along the intervals $]b_{i},b_{i+1}[$ for $i=0,\ldots,r$. In
 particular there exists $\delta$ such that $\eps>0$ and $\eps'>0$ belong
 to the same interval if $\eps<\delta$ and $\eps'<\delta$.
\end{proof}

\begin{rmk} If $X$ has at most an isolated singularity at $a$, then
  Theorem \ref{CS} asserts that the link is a Nash invariant, and
  therefore its virtual Poincar\'e polynomial is well defined by
  Proposition \ref{beta-nash}. If the singular
  component containing $a$ is a curve, then the link of $X$ at $a$ has an
  isolated singularity and the link is a blow-Nash invariant by
  Theorem IIb in \cite{Koike}. This implies again that its virtual Poincar\'e polynomial is well defined.

In the general case, the question of the existence of continuous
moduli for blow-Nash equivalence remains open, and Corollary \ref{link} says only that the virtual Poincar\'e polynomial of the link is well defined.
\end{rmk}



\begin{proof}[Proof of Theorem \ref{main}] We reduce the proof to the
 proper case using the graph of $p$. In the proper case, we express
 the virtual Poincar\'e polynomial of the fibres of $p$ in terms of a
 resolution of the singularities of $p$. Then, we use Theorem \ref{CS}
 and Theorem \ref{FKS} to obtain Nash triviality results along the
 resolution of $p$ in the smooth and in the normal crossing case.\\

\textbf{Reduction to the proper case}

We first reduce the proof of theorem \ref{main} to the case of a
proper map $f$.
Let $G\subset \R^n\times \R^m$ denote the graph of $f$ and consider its Zariski closure
$\overline G$ in $\PP^n(\R)\times \R^m$. Then the projection
$p:\overline G \to Y$ together with its restriction to $\overline G
\setminus G$ are proper regular maps. 

By additivity of the virtual Poincar\'e polynomial, it suffices now to
prove theorem \ref{main} for $p:\overline G \to Y$ and $p:\overline
G\setminus G \to Y$.\\

\textbf{The proper case}

We assume that $f$ is proper.
Choose a resolution of the singularities of $X\subset \R^n$ as in
Theorem \ref{hiro}. Denote by $D^{j+1}$ the exceptional divisor $D^{j+1}=(\pi^j)^{-1}(C^j)$ of
the $j$-th blowing-up of $M^j$ along $C^j$, by $\tilde E^j$ the
intersection $\tilde E^j=X^j\cap E^j$
of the exceptional divisor $E^j$ with the strict transform $X^j$, and
by $\tilde D^{j+1}$ the intersection $\tilde D^{j+1}=X^{j+1} \cap D^{j+1}$. Denote by
$\sigma^j$ the composition map
$$\sigma^j=f\circ \pi^1 \circ \cdots \circ \pi^{j} :M^{j+1}\to Y,$$
which is proper.
Let  $X_y^{j+1}$ denote the fibre over $y\in Y$ of the restriction of $\sigma^j$ to
$X^{j+1}$, namely
$$X_y^{j+1}=(\sigma^j_{|X^{j+1}})^{-1}(y).$$
Denote similarly by $C_y^{j+1}$, $\tilde D_y^{j+1}$ and $\tilde E_y^{j+1}$ the
fibres over $y\in Y$ of the restrictions of $\sigma^j$ to $C^{j+1}$, $\tilde D^{j+1}$ and
$\tilde E^{j+1}$ respectively. 

\vskip 3mm
The proof proceeds by induction on the number of blowings-up. We detail the first two blowings-up in order to illustrate the difference in the use of Theorems \ref{CS} and \ref{FKS}.

\vskip 3mm
\textit{First blowing-up}. 

Consider the blowing-up of $M^1=\R^n$ with centre $C^1$. By Theorem \ref{CS}, there exists $Z^1\subset Y$ of
codimension at least one in $Y$ such that the restriction ${f}_{|C^1}$ of $f$ to $C^1$ is Nash trivial along any connected component of
$Y\setminus Z^1$, and $C^1_y$ is smooth for any $y\in Y\setminus Z^1$,
since $C^1$ is non singular (note that the fibres might
be empty in
case $f$ is not a submersion). In particular for $a,b\in Y\setminus
Z^1$ in the same connected component of $Y\setminus Z^1$, the virtual
Poincar\'e polynomial of $C_a^1$ and $C_b^1$ coincide by Proposition \ref{beta-nash} since $C_a^1$
and $C_b^1$ are Nash diffeomorphic compact real algebraic sets.

Even if $E^2$ is smooth in $M^2$, its intersection $\tilde E^2$ with $X^2$ is no
longer smooth in general, and we can not apply Theorem \ref{CS} to find a Nash
trivialisation of $\tilde E^2$ at this stage.

At the level of virtual Poincar\'e polynomials, by additivity one obtains
$$\beta(X^1_y)=\beta(C^1_y)+\beta(X^2_y)-\beta(\tilde E^2_y)$$
for $y\in Y$, where moreover $\beta(C^1_a)=\beta(C^1_b)$ for $a,b$ in the same connected component of $Y\setminus Z^1$.

\vskip 3mm
\textit{Second blowing-up}.

Consider the blowing-up of $M^2$ with centre $C^2$. As $C^2$ is non singular, there exists $Z^2_1\subset Y$ of codimension at least one in $Y$ such that the
restriction $\sigma^2_{|C^2}$ of $\sigma^2$ to $C^2$ is Nash trivial
along $Y\setminus Z^2_1$ by Theorem \ref{CS}. As a consequence $\beta(C^2_a)=\beta(C^2_b)$ for
$a,b \in Y\setminus Z^2_1$ by Proposition \ref{beta-nash}.

This blowing-up gives rise the following commutative diagrams
$$\xymatrix{ & \tilde D^3_y \ar[r]\ar[d]_{\pi^2_y} &  X^3_y \ar[d]^{\pi^2_y}
 &&&&     D^3_y \cap (\pi^2)^{-1}(\tilde E^2_y)
 \ar[r]\ar[d]_{\pi^2_y} &    (\pi^2_y)^{-1}(\tilde E^2_y)  \ar[d]^{\pi^2_y}   \\
   &       C^2_y \ar[r]  & X^2_y                      &&&&
          C^2_y\cap E^2_y    \ar[r]  &   \tilde E^2_y   }$$
for $y\in Y$ (note that $C^2_y\cap E^2_y= C^2_y\cap \tilde E^2_y$
since $C^2$ is included in $X^2$). Note that $C^2\cap E^2$ is non singular since $C^2$ and $E^2$ have simultaneously normal crossings, so we can apply Theorem \ref{CS} to obtain a Nash trivialisation of the intersection
$C^2\cap E^2$ outside a codimension at least one subset $Z^2_2 \subset
Y$ of $Y$. In particular the virtual Poincar\'e polynomials of $C^2_a\cap E^2_a$ and $C^2_b\cap E^2_b$ coincide for
$a,b$ in the same connected component of $Y\setminus Z^2_2$ by Proposition \ref{beta-nash}.

Note moreover that $(\pi^2)^{-1}(\tilde E^2_y) \setminus \big (D^3 \cap
(\pi^2)^{-1}(\tilde E^2_y)\big )$ is isomorphic to $ \tilde E^3 \setminus
\tilde D^3$ since $\pi^2$ is an isomorphism from $M^3 \setminus D^3$ to $M^2 \setminus C^2$.
Therefore at the level of virtual Poincar\'e polynomials
$$\beta((\pi^2)^{-1}(\tilde E^2_y))-\beta(D^3 \cap (\pi^2)^{-1}(\tilde E^2_y))=\beta(\tilde E^3)-\beta(\tilde D^3).$$
In particular we obtain the following expression for the virtual
Poincar\'e polynomial of $X^1_y$:
$$\beta(X^1_y)=\beta(C^1_y)+\beta(C^2_y)-\beta(C^2_y\cap
E^2_y)+\beta(X^3_y)-\beta(\tilde E^3_y),$$
with $\beta(C^1_a)=\beta(C^1_b)$, $\beta(C^2_a)=\beta(C^2_b)$ and $\beta(C^2_a\cap E^2_a)=\beta(C^2_b\cap E^2_b)$ for
$a,b$ in the same connected component of $Y\setminus Z^2$, where $Z^2=Z^1 \cup Z_1^2 \cup Z^2_2$.

\vskip 3mm
If $s\geq 4$, assume that after $j$ blowings-up, with $j\in \{2,\ldots,s-1\}$, we are in such a situation that 
$$\beta(X^1_y)=\sum_{i=1}^{j} \big (\beta(C^i_y)-\beta(C^i_y\cap
E^i_y)\big )+\beta(X^{j+1}_y)-\beta(\tilde E^{j+1}_y),$$
with $\beta(C^i_a)=\beta(C^i_b)$ and $\beta(C^i_a\cap E^i_a)=\beta(C^i_b\cap E^i_b)$ for $a,b$ in the same connected component of $Y\setminus Z^{j}$ and
$i\in \{1,\ldots,j\}$, where $Z^{j} \subset Y$ is a subset of
codimension at least one in $Y$. 

We proceed similarly to the second blowing-up. As $C^{j+1}\subset M^{j+1}$ is non singular, there exists $Z^{j+1}_1\subset Y$ of codimension at least one in $Y$ such that the
restriction $\sigma^{j+1}_{|C^{j+1}}$ of $\sigma^{j+1}$ to $C^{j+1}$ is Nash trivial
along $Y\setminus Z^{j+1}_1$ by Theorem \ref{CS}. As a consequence $\beta(C^{j+1}_a)=\beta(C^{j+1}_b)$ for
$a,b \in Y\setminus Z^{j+1}_1$ by Proposition \ref{beta-nash}.

This blowing-up gives rise the following commutative diagrams
$$\xymatrix{ & \tilde D^{j+2}_y \ar[r]\ar[d]_{\pi^{j+1}_y} &  X^{j+2}_y \ar[d]^{\pi^{j+1}_y}
 &&     D^{j+2}_y \cap (\pi^{j+1})^{-1}(\tilde E^{j+1}_y)
 \ar[r]\ar[d]_{\pi^{j+1}_y} &    (\pi^{j+1}_y)^{-1}(\tilde E^{j+1}_y)  \ar[d]^{\pi^{j+1}_y}   \\
   &       C^{j+1}_y \ar[r]  & X^{j+1}_y                      &&
          C^{j+1}_y\cap E^{j+1}_y    \ar[r]  &   \tilde E^{j+1}_y   }$$
for $y\in Y$ (note that $C^{j+1}_y\cap E^{j+1}_y= C^{j+1}_y\cap \tilde E^{j+1}_y$
since $C^{j+1}$ is included in $X^{j+1}$). The difference with the case of the second blowing-up is now that $C^{j+1}\cap E^{j+1}$ is not
necessarily non singular since $E^{j+1}$ has normal crossings, so
we can not apply directly Theorem \ref{CS}. However, as $C^{j+1}$ is smooth and has simultaneously only normal
crossings with $E^{j+1}$ by Theorem \ref{hiro}, we may apply Theorem \ref{FKS} with $M^{j+1}$, $C^{j+1}$
and $E^{j+1}$ in order
 to obtain a Nash trivialisation of the intersection
$C^{j+1}\cap E^{j+1}$ outside a codimension at least one subset $Z^{j+1}_2 \subset
Y$ of $Y$. In particular the virtual Poincar\'e polynomials of $C^{j+1}_a\cap E^{j+1}_a$ and $C^{j+1}_b\cap E^{j+1}_b$ coincide for
$a,b$ in the same connected component of $Y\setminus Z^{j+1}_2$ by Proposition \ref{beta-nash}.

Finally, we obtain by additivity of the virtual Poincar\'e polynomial, in a similarly way than in the case of the second blowing-up, that 
$$\beta(X^1_y)=\sum_{i=1}^{j+1} \big (\beta(C^i_y)-\beta(C^i_y\cap
E^i_y)\big )+\beta(X^{j+2}_y)-\beta(\tilde E^{j+2}_y),$$
with $\beta(C^i_a)=\beta(C^i_b)$ and $\beta(C^i_a\cap E^i_a)=\beta(C^i_b\cap E^i_b)$ for $a,b$ in the same connected component of $Y\setminus Z^{j+1}$ and
$i\in \{1,\ldots,j+1\}$, where $Z^{j+1}=Z^{j+1}_1 \cup Z^{j+1}_2 \subset Y$ is a subset of
codimension at least one in $Y$.

By induction, we obtain after the last blowing-up that the virtual Poincar\'e polynomial of $X^1_y$ is equal to:
$$\beta(X^1_y)=\sum_{i=1}^{s-1} \big (\beta(C^i_y)-\beta(C^i_y\cap
E^i_y)\big )+\beta(X^s_y)-\beta(\tilde E^s_y),$$
with $\beta(C^i_a)=\beta(C^i_b)$ and $\beta(C^i_a\cap E^i_a)=\beta(C^i_b\cap E^i_b)$ for $a,b$ in the same connected component of $Y\setminus Z^{s-1}$ and
$i\in \{1,\ldots,s-1\}$, where $Z^{s-1} \subset Y$ is a subset of
codimension at least one in $Y$.

The singularities of the variety $X$ being resolved in $X^s$, we are in position to apply once more Theorem \ref{CS} to
$\sigma ^s_{|X^s}:X^s\to Y$ and Theorem \ref{FKS} to $M^s$, $\tilde
E^s$ and $\sigma_s$ since $\tilde E^s$ has only normal crossings. We
obtain that way a new subset $Z^s\subset Y$ of codimension at least
one in $Y$ such that $\beta(X^s_a)=\beta(X^s_b)$ and $\beta(\tilde
E^s_a)=\beta(\tilde E^s_b)$ for $a,b$ in the same connected component of $Y\setminus Z^s$.

Finally, for $a,b$ in the same connected component of $Y\setminus Z$, where $Z=Z^{s-1}\cup Z^s$, we obtain that $\beta(X^1_a)=\beta(X^1_b)$ and the proof is achieved.
\end{proof}

\enddocument